\documentclass[reqno]{amsart}
 
\usepackage{amsmath}
\usepackage{amsrefs}
\usepackage{amssymb}
\usepackage{mathrsfs}
\usepackage{enumitem}

\def\Z{\mathbb{Z}}
\def\S{\mathcal{S}}

\def\om{\omega}

\newcommand{\be}{\begin{equation}}
\newcommand{\ee}{\end{equation}}
\newcommand{\ba}{\begin{align}}
\newcommand{\ea}{\end{align}}

\newcommand{\abs}[1]{\lvert#1\rvert}

\DeclareMathOperator{\GA}{GA}

\newtheorem{theorem}{Theorem}[section]

\newtheorem{lemma}{Lemma}[section]

\title[Group determinants for $\mathbb{Z}_p^n$ ]{The Integer group determinants for  $\mathbb{Z}_p^n$}

\author[M. Mossinghoff]{Michael J.  Mossinghoff}
\address{Center for Communications Research, Princeton, NJ, USA}
\email{m.mossinghoff@ccr-princeton.org}

\author[C. Pinner]{Christopher Pinner}
\address{ Department of Mathematics\\
         Kansas State University\\
         Manhattan, KS 66506, USA}
\email{pinner@math.ksu.edu}

\date{\today}

\begin{document}

\begin{abstract}
We  give new conditions on allowable powers of a prime $p$ dividing an integer group determinant for the group $\mathbb{Z}_p^n$, and show these conditions are sharp for $n\leq 5$.
The values coprime to $p$ for these groups are already known; determining which multiples of allowable powers of $p$ occur is much more complicated.
We show that all multiples of sufficiently large powers of $p$ occur as integer group determinants of $\mathbb{Z}_p^n$.
We also provide a complete characterization for the group $\mathbb{Z}_3^3$, where particular arithmetic conditions are required for multiples of certain powers of $3$.
\end{abstract}

\maketitle

\section{Introduction}\label{secIntroduction}

Recall that the \textit{group determinant} of a finite group $G=\{g_1,\ldots,g_n\}$ is a homogeneous polynomial of degree $n$ with integer coefficients in $n$ variables $\{x_1, \ldots, x_n\}$, defined as the determinant of an $n\times n$ matrix whose entry at row $i$ and column $j$ is $x_k$, where $g_i g_j^{-1} = g_k$ in $G$.
We denote the group determinant of $G$ by
\[
\mathscr{D}_G(x_1,\ldots,x_n).
\]
It is well known that the group determinant determines the group \cite{FormanekSibley}.
At an AMS meeting in 1977, Olga Taussky-Todd \cite{OTT} asked about the values achieved by the group determinant when it is evaluated at values in $\mathbb{Z}^n$.
She was particularly interested in the case when $G$ is a finite cyclic group, so in this case the question is to determine for a fixed $n$ all possible values of determinants of integer $n\times n$ circulant matrices.
Let $\mathcal{S}(G)$ denote the set of such values associated with a group $G$; we say $\mathcal{S}(G)$ is the set of \textit{integer group determinants} of $G$.
It is convenient to partition $\mathcal{S}(G)$ as
\[
\S(G)= \S^*(G)\cup \S_0(G),
\]
where $\S^*(G)$ are the values coprime to $\abs{G}=n$ and  $\S_0(G)$ are those divisible by a prime dividing $n$.

A number of results concerning $\mathcal{S}(G)$ are known for various groups $G$.
Much prior research falls into one of two broad categories: determining the smallest nontrivial value of $\mathcal{S}(G)$, and determining information on the set $\mathcal{S}(G)$, including characterizing its values.
Our work here focuses on the latter category, but we provide a brief overview of both topics here.

The first category is connected to an analogue for group rings of a well-known problem in analysis and number theory known as Lehmer's problem regarding the Mahler measure of a polynomial \cite{Lehmer}.
For a finite abelian group, the group determinant splits into linear factors, and its value coincides with Lind's 2005 generalization \cite{Lind} of the Mahler measure to an arbitrary compact abelian group.
This is defined in terms of the characters of the group, with the classical Mahler measure corresponding to the group $\mathbb R/\mathbb{Z}$, see for example \cites{Kaiblinger10,Kaiblinger12}.
For $F=\sum_{g\in G} a_g g \in \mathbb{Z}[G]$, the Lind--Mahler measure of $F$ is an integer group determinant for $G$,
\[
M_G(F) := \mathscr{D}_G(a_{g_1},\ldots ,a_{g_n}).
\]
The Lind--Lehmer problem for the group then corresponds to finding the smallest nontrivial element of $\S(G)$.
Results of this type for the groups $\mathbb{Z}_p \times \mathbb{Z}_{p^2}$, $\mathbb{Z}_p^k$, $\mathbb{Z}_2^r \times \mathbb{Z}_4^s$, $\mathbb{Z}_m$, and $\mathbb{Z}_m\times\mathbb{Z}_p^n$ with $p\nmid m$ can be found respectively in \cites{pgroups,DeSilvaPinner,MPP19,Pigno1,PPV16}, for various $3$-groups in \cite{CP3}, and for general finite abelian groups in \cite{MahoneyNewman}.
Here and throughout this article, we write $\mathbb{Z}_n$ for the cyclic group of order $n$, and $p$ denotes an odd prime number.

This realization of the Lind--Mahler measure in this setting as an integer group determinant allows extending the concept of this measure to non-abelian finite groups in a natural way.
The elements of $\mathbb{Z}[G]$  can be associated with elements from a  polynomial ring, with non-commuting monomials in the case of a non-abelian group, and the usual multiplicative property $M_G(f_1 f_2)=M_G(f_1)M_G(f_2)$ of the Mahler measure holds if we multiply and reduce polynomials using the group relations on the monomials.
The Lind--Lehmer problem for the group $G$ then becomes to determine
\begin{equation*}\label{LindLehmer}
\lambda(G):= \frac{1}{|G|} \log \min \left\{ \abs{M_G(F)} \geq 2 : F\in \mathbb{Z} [G]\right\}.
\end{equation*}
 
Work in the second category, regarding characterizing the values of $\mathcal{S}(G)$, dates to 1980, when Laquer \cite{Laquer} and Newman \cite{Newman1} proved that for a finite cyclic group $G$ all values coprime to $\abs{G}$ can be achieved as integer group determinants:
\begin{equation}\label{eqnSstarCyclic}
\mathcal{S}^*(\mathbb{Z}_n) = \{m\in \mathbb{Z} : \gcd(m,n)=1 \}.
\end{equation}
They also showed that some values do not occur in $\mathcal{S}_0(\mathbb{Z}_n)$: if $p^t \parallel n$, $p\mid m$ and $m\in\mathcal{S}_0(\mathbb{Z}_n)$, then necessarily $p^{t+1} \mid m$.
They proved that this condition is sufficient as well when $n=p$ or $n=2p$, so they obtain a full characterization in these cases:
\begin{equation}\label{newman}
\begin{split}
\mathcal{S}_0(\mathbb{Z}_p) &= p^2\mathbb{Z},\\
\mathcal{S}_0(\mathbb{Z}_{2p}) &= 4\mathbb{Z} \cup p^2\mathbb{Z}.
\end{split}
\end{equation}
Newman \cite{Newman2} proved that this restriction also suffices to characterize the case $n=9$,
\[
\mathcal{S}_0(\mathbb{Z}_9) = 27\mathbb{Z},
\]
but that it is not sufficient in general, proving that
\begin{equation}\label{eqnNewmanRestric}
p^{t+1}\not\in\mathcal{S}_0(\mathbb{Z}_p^t)
\end{equation}
for any $t\geq2$ when $p\geq5$.
In general, Newman showed that $\S_0(\mathbb{Z}_{p^t}) \subseteq p^{t+1}\mathbb{Z}$ for $t\geq2$.

In 2023, the authors \cite{MP23} proved that the smallest power of $p$ in $\mathcal{S}_0(\mathbb{Z}_{p^t})$ is always $p^{2t}$, so that
\[
p^{2t}\mathbb{Z}\subseteq \mathcal{S}_0(\mathbb{Z}_{p^t}) \subseteq p^{t+1}\mathbb{Z},
\]
and established some necessary conditions on integers $m$ so that $m p^j \in \mathcal{S}_0(\mathbb{Z}_{p^t})$ when $t<j<2t$.
While restrictive, those conditions are not expected to be sufficient, and in evidence of this in the same paper the authors determine characterizations of the sets $\mathcal{S}(\mathbb{Z}_{25})$ and $\mathcal{S}(\mathbb{Z}_{27})$, which are somewhat complicated to describe in full here.
Complete characterizations for $\mathcal{S}(G)$ are also known for all groups of order at most $19$ \cites{PP15,smallgps,YYC8C2,YYCirc16,YYC42,YYC24,YYC422,PPQ16,YYNA16,YY16,BPP24,BoerkoelPinner,PaudelPinnerZnxH,PaudelPinner25}, as well as for $S_4$ \cite{PinnerS4} and the abelian groups of order $24$ \cite{Panraksa24}.
The set $\mathcal{S}(G)$ has also been investigated for metacyclic groups \cite{Mahoney}, dihedral groups $D_{2n}$ \cite{BoerkoelPinner}, dicyclic groups $Q_{4n}$ \cite{dicyclic}, the Heisenberg groups $H(\mathbb{Z}_p)$ \cite{Heisenberg}, and the general affine group $\GA(1,p^k)$ of degree $1$ over the finite field $\mathbb{F}_{p^k}$ \cites{BPP25,OP26}.
In particular, these papers provide significant advances with respect to the three non-abelian groups of order $20$: complete characterizations are obtained in these respective works for $\mathcal{S}(D_{20})$ and $\mathcal{S}(\GA(1,5))$, and partial information is obtained regarding $\mathcal{S}(Q_{20})$.

The set $\mathcal{S}(G)$ of integer group determinants has also been studied for a number of families of noncyclic abelian $p$-groups $G$ \cites{Kaiblinger12,PaudelPinnerZnxH,YYLaquer23}.
Recently, complete characterizations were determined by Panraksa for $G=\mathbb{Z}_5^2$ and $\mathbb{Z}_7^2$ \cites{Panraksa25,Panraksa49}.

The present article forms a companion to \cite{MP23} by studying the integer group determinants of certain noncyclic abelian $p$-groups, especially $\mathbb{Z}_p^n$.
We determine improved restrictions on the permissible powers $p^j$ for which integers of the form $m p^j$ with $p\nmid m$ occur in $\mathcal{S}_0(\mathbb{Z}_p^n)$ for $n\geq2$, and show that these bounds are sharp for $n\leq 5$.
We also find similar conditions for other products of cyclic $p$-groups, such as $\mathbb{Z}_p^t\times \mathbb{Z}_p^{n-1}$ and $\mathbb{Z}_p^s\times \mathbb{Z}_p^t$.
In addition, we study certain special cases in greater detail, and determine a complete characterization for the first time of $\mathcal{S}(\mathbb{Z}_3^3)$.
We also note some further improvements for $\mathcal{S}_0(\mathbb{Z}_{11}^2)$ and $\mathcal{S}_0(\mathbb{Z}_{13}^2)$.

This paper is organized in the following way.
Section~\ref{dilum} describes our results for $\S_0(\Z_p^n)$.
Section~\ref{higherpowers} states some additional results for $\mathcal{S}_0(G)$ and $\mathcal{S}^*(G)$ where $G$ is a product of $p$-groups involving one or more components $\Z_{p^t}$ with $t\geq2$.
Section~\ref{sec55} presents additional information on $\S_0(\Z_p^2)$ for a number of small primes $p$, and Section~\ref{sec333} gives detailed results for $\S_0(\Z_3^3)$.
Proofs of the theorems presented in Sections~\ref{dilum}, \ref{higherpowers}, \ref{sec55} and~\ref{sec333} appear respectively in Sections~\ref{secPfs2}, \ref{secPfs3}, \ref{secPfs4} and~\ref{secPfs5}.

\section{The groups $\mathbb{Z}_p^n$}\label{dilum}

For $G=\mathbb{Z}_p^n$ the group determinant $\mathscr{D}_G(a_{g_1},\ldots ,a_{g_{|G|}})$  takes the form
\[
M_G(F)=\prod_{j_1=0}^{p-1} \cdots \prod_{j_n=0}^{p-1} F(w^{j_1},\ldots ,w^{j_n}),\;\; w:=e^{2\pi i/p},
\]
where
\[
F(x_1,\ldots ,x_n) = \!\!\!\!\! \sum_{g=(m_1,\ldots ,m_n)\in G} \!\!\!\!\! a_g\: x_1^{m_1}\cdots x_n^{m_n} \in \mathbb{Z} [x_1,\ldots ,x_n].
\]
From \cites{smallgps,YYC24} we have
\begin{align*} 
 \S_0(\mathbb{Z}_2^2) &= \{2^a(2m+1): a=4 \text{ or } a\geq 6,\; m\in \mathbb{Z}\}, \quad \S^*(\mathbb{Z}_2^2) =\{4m+1 : m\in \mathbb{Z}\}, \\
 \S_0(\mathbb{Z}_2^3) &= \{2^8(4m+1) \hbox{ or } 2^{12}m : m\in \mathbb{Z}\}, \quad \S^*(\mathbb{Z}_2^3) = \{8m+1 : m\in \mathbb{Z}\},\\
 \S_0(\mathbb{Z}_2^4) &= \{2^{16}(4m+1) \hbox{ or }  2^{26}m : m\in \mathbb{Z}\} \cup\{ 2^{24}m : m \in A_{24}\},\\
 \S^*(\mathbb{Z}_2^4) &= \{16m+1 : m\in \mathbb{Z}\}, 
\end{align*}
where
\[
A_{24} = \{4k+1 : k\in \mathbb{Z}\} \cup
\{8k+3 : k\in \mathbb{Z}\} \cup
\{(8k-3)(8\ell+3) : k,\ell\in \mathbb{Z}\},
\]
and
\be \label{p=3}  \S_0(\mathbb{Z}_3^2)=3^6\mathbb{Z}, \quad \S^*(\mathbb{Z}_3^2)=\{m : m\equiv \pm 1 \bmod 9\}, \ee
but no other $\S(\mathbb{Z}_p^n)$ with $n\geq 2$  had been completely determined.

In \cite{DeSilvaPinner} a complete description was obtained for the integer group determinants coprime to $p$:
\begin{align}\label{coprimevalues}
\S^*(\mathbb{Z}_p^n) & =\{ x\in \mathbb{Z} : x^{p-1}\equiv 1 \bmod p^n\}\nonumber\\
 & = \{m\in \mathbb{Z} : m\equiv a^{p^{n-1}} \bmod p^n, \text{ for some } 1\leq a <p\}.
 \end{align}
For example,
\be \label{3coprime}  \S^*(\mathbb{Z}_3^n)=\{m : m\equiv \pm 1 \bmod 3^n\}.\ee
For the multiples of $p$, it gave only the rough bound
\be \label{oldp} \S_0(\mathbb{Z}_p^n) \subseteq p^{\frac{p^n-1}{p-1}+1}  \mathbb{Z}, \ee
the main interest there being the smallest non-trivial measure, the values divisible by $p$ being greater than the trivial  bound $|G|-1.$
Obtaining \eqref{oldp} is at least straightforward: writing $\pi =1-w$ we have $|\pi|_p=p^{-1/(p-1)}$ and $F(w^{j_1},\ldots ,w^{j_n})\equiv F(1,\ldots ,1) \bmod \pi$, and $p\mid M_G(F)$
if and only if $p\mid F(1,\ldots ,1)$, with $\pi$ dividing all the other terms, and
\be \label{rough}  p\mid M_G(F)\;\; \Rightarrow \;\;p^{\frac{p^n-1}{p-1}+1}  \mid M_G(F). \ee
From \eqref{newman} this is sharp when $n=1$.
In \cite{Heisenberg} the exponent in \eqref{rough} was improved for $n=2$.
 We show here that the exponent can be improved for all $n\geq 2$.

For $p\geq 5$ we obtain a sharp result for $2\leq n\leq 5$.

\begin{theorem}\label{Zpp}
Suppose that $p\geq 5$ and  $G=\mathbb{Z}_p^n$ with $n\geq 2$ and $F$ in $\mathbb{Z} [x_1,\ldots ,x_n]$.
\begin{enumerate}[label=(\roman*)]
\item\label{Zpp1}
For $n=2$ we have
\be \label{n=2}   p\mid M_G(F) \Rightarrow p^{p+3} \mid M_G(F). \ee
Moreover for any $k\geq 0$ there is an $F$ with $p^{p+3+k}\parallel M_G(F)$.
\item\label{Zpp2}
For $n=3$ we have
\be \label{n=3}   p\mid M_G(F) \Rightarrow p^{p^2+2p+3} \mid M_G(F). \ee
Moreover for any $k\geq 0$ there is an $F$ with $p^{p^2+2p+3+k}\parallel M_G(F)$.
\item\label{Zpp3}
For $n\geq 4$ we have
\be \label{n>3} p\mid M_G(F)\;\; \Rightarrow \;\;p^{\frac{p^n-1}{p-1} +\frac{p^{n-1}-1}{p-1} +\frac{p^{n-3}-1}{p-1}+1}\mid M_G(F). \ee
For $n=4$ or $5$  this is sharp, that is, for any $k\geq0$, for $n=4$ there is an $F$ with $p^{p^3+2p^2+2p+4+k}\parallel M_G(F)$, and for $n=5$ there is an $F$ with $p^{p^4+2p^3+2p^2+3p+4+k} \parallel M_G(F)$.
\end{enumerate}
\end{theorem}

For $p=3$ we have the optimal exponent for all $n$, with gaps that can be made precise for $2\leq n\leq 5$.

\begin{theorem}\label{pp=3}
Suppose that $G=\mathbb{Z}_3^n$ with $n\geq 2$ and $F\in\mathbb{Z} [x_1,\ldots ,x_n]$, then 
$$3\mid M_G(F)\;\; \Rightarrow \;\;3^{2 \cdot 3^{n-1}} \mid M_G(F). $$
This is best possible with $M_G(2+x_1^2)=3^{2\cdot 3^{n-1}}$.
In addition:
\begin{enumerate}[label=(\roman*)]
\item\label{pp=31}
For $n=3$ there is an $F$ with $3^a\parallel M_G(F)$ for any $a\geq 18$.
\item\label{pp=32}
For $n=4$ the multiples of $3$ have
\[
3^a \parallel M_G(F)\;  \Rightarrow \; a=54 \hbox{ or } a\geq 56
\]
with all such $a$ achievable.
\item\label{pp=33}
For $n=5$ the multiples of $3$ have 
\[
3^a \parallel M_G(F) \;   \Rightarrow \; a=162 \hbox{ or } a\geq 167
\]
with all such $a$ achievable.
\item\label{pp=34}
For $n\geq 6$ the multiples of $3$ have
\[
3^a \parallel M_G(F) \;  \Rightarrow \; a=2\cdot 3^{n-1} \hbox{ or } a\geq 2\cdot 3^{n-1} +\frac{1}{2}(3^{n-3}+1).
\]
\end{enumerate}
\end{theorem}

In general, we have $M_{\mathbb{Z}_p}(1-x+\Phi_p(x))=p^2$ and so $p^{2p^{n-1}}\in \S_0(\mathbb{Z}_p^n)$
where here and throughout $\Phi_p(x)$ denotes the $p$th cyclotomic polynomial
\[
\Phi_p(x) :=1+x+\cdots +x^{p-1} = \frac{x^p-1}{x-1}.
\]

Although we have a precise determination of the powers of $p$ that can divide group determinants, it is less easy to say which multiples of those powers can occur.
We consider a few small cases in the next sections.
For $\mathbb{Z}_5^2$ all multiples occur, but for $G=\mathbb{Z}_3^3$ the situation becomes much more complicated.
Dealing with the general case does not seem feasible. 

Finally, for suitably large powers of $p$ we observe that we obtain all multiples of that power.

\begin{theorem} \label{allmultiples} For $n\geq 2$,
\[
p^{2p^{n-1}} \S\left(\mathbb{Z}_p^{n-1}\right) \subseteq \S_0\left(\mathbb{Z}_p^{n}\right).
\]
In particular,
\[
p^{2+2p+2p^2+\cdots + 2p^{n-1}} \mathbb{Z} \subset \S_0\left(\mathbb{Z}_p^n\right).
\]
In addition, $ \S_0\left(\mathbb{Z}_p^n\right)$ contains  anything of the form
\begin{gather*}
p^{2p+2p^2+\cdots + 2p^{n-1}}m, \quad \gcd(m,p)=1,\\
p^{2p^{n-1}}a, \quad a^{p-1}\equiv 1 \bmod p^{n-1}.
\end{gather*}
\end{theorem}

For example, we have
\be \label{allp} p^{2p+2}\mathbb{Z} \subset \S_0(\mathbb{Z}_p^2) \subseteq p^{p+3}\mathbb{Z}, \quad p^{2p^2+2p+2}\mathbb{Z} \subset \S_0(\mathbb{Z}_p^3) \subseteq p^{p^2+2p+3}\mathbb{Z}. \ee

\section{Higher powers of $p$}\label{higherpowers}

In Theorems \ref{Zpp} and \ref{pp=3}, we can readily replace one of the factors $\mathbb{Z}_p$ with $\mathbb{Z}_{p^t}$.

\begin{theorem}\label{withpower}
Suppose that  $G=\mathbb{Z}_{p^t} \times \mathbb{Z}_p^{n-1}$ with $n\geq 2$ and $t\geq 1$, then
\[
p \mid M_G(F) \;\; \Rightarrow \;\;p^{\alpha + (t-1) p^{n-1}} \mid M_G(F),
\]
where $\alpha$ is the exponent for $\mathbb{Z}_p^n$ in Theorem~\ref{Zpp} or~\ref{pp=3}.
For $p=3$ the exponent $(t+1)3^{n-1}$ is sharp for all $n\geq 2$.
For $p\geq 5$ the exponent is sharp for $2\leq n\leq 5$.
\end{theorem}

We also obtain the following generalization of \eqref{allp}.

\begin{theorem}\label{genallp}
Suppose that $G=\mathbb{Z}_{p^t} \times \mathbb{Z}_p$ with  $t\geq 2.$
\begin{enumerate}[label=(\roman*)]
\item\label{genallp1}
If $p\geq 5,$ or $p=3$ with $t\geq 3$, then $p^{2tp} \in \S_0(G)$ and
\[
p^{2tp+2}\mathbb{Z} \subseteq \S_0 \left(  G \right) \subseteq p^{tp+3}\mathbb{Z}.
\]
Further, the exponent on $p$ in the upper inclusion is sharp.
\item\label{genallp2}
If $p^{tp+3}m\in \S_0(G)$, $p\nmid m,$ then $q^a\mid m$ for some prime power $q^a\equiv 1 \bmod p^{t-1}$.
In particular, $p^{tp+3}\not \in \S_0(G)$.
\item\label{genallp3}
For $G=\mathbb{Z}_9 \times \mathbb{Z}_3$ we have
$$ 3^{11}\mathbb{Z} \subseteq S_0(G) \subseteq 3^9 \mathbb{Z}, \quad M_G(1+x+x^4)=3^9. $$
\item\label{genallp4}
More generally, for $G=\mathbb{Z}_9 \times \mathbb{Z}_3^{n-1}$ we have
\be \label{lower3s}  3^{5\cdot 3^{n-1}-3}\mathbb{Z} \subseteq \S_0(G)\subseteq 3^{3^n} \mathbb{Z}, \quad M_G(1+x+x^4)=3^{3^n}. \ee
\end{enumerate}
\end{theorem} 

If we want to replace more than one component then the situation becomes rapidly more complicated. We obtain the optimal result for two components.

\begin{theorem}\label{twocomponents}
Suppose that $G=\mathbb{Z}_{p^s}\times \mathbb{Z}_{p^t}$, with $1\leq s\leq t$.
Then
\[
p\mid M_G(F) \;\; \Rightarrow \;\; p^{\alpha}\mid M_G(F), \;\; \alpha=(t+1-s) p^s +3\left(\frac{p^s-1}{p-1}\right).
\]
This is best possible: for any $k\geq 0$ there is an $F$ with $p^{\alpha +k}\parallel M_G(F).$
\end{theorem}
Analogous to Theorem \ref{genallp}, for $G=\mathbb{Z}_{p^t}\times \mathbb{Z}_{p^s}$, $t\geq s$, we  have
\be \label{lowerks}  p^{2tp^s+2s}\mathbb{Z} \subseteq \S_0(G) \subseteq p^{(t+1-s)p^s +3(1+\cdots +p^{s-1})}\mathbb{Z} ,\quad p^{2tp^s}\in \S_0(G). \ee

To complement these results regarding $\mathcal{S}_0(G)$ for a number of non-elementary and non-cyclic $p$-groups $G$, it is natural to ask about $\mathcal{S}^*(G)$ as well.
For the case $G=\mathbb{Z}_p^n$, the simple characterization is recorded in \eqref{coprimevalues}, and the cyclic case is answered in \eqref{eqnSstarCyclic}, but the situation is quite different for more complicated $p$-groups.
To illustrate this, we determine this set for $G=\mathbb{Z}_3\times\mathbb{Z}_9$.
The following result then forms a companion to Theorem~\ref{genallp}\ref{genallp3}.

\begin{theorem}\label{thmSstarZ3Z9}
The set $\mathcal{S}^*(\mathbb Z_3\times \mathbb Z_9)$ consists of the integers $D \equiv \pm 1\bmod 27$ together with the integers $D\equiv\pm 8$ or $\pm 10 \bmod 27$ for which
\begin{equation}\label{nice}
D = AN(B+C\om)=A(B^2-BC+C^2),
\end{equation}
where $\omega=e^{2\pi i/3}$ and $A$, $B$, and $C$ are integers satisfying $A\equiv \pm 2$ or $\pm 4 \bmod 9$ with
\[
B\equiv A \bmod 9, \;\; C\equiv A-B \bmod 27.
\]
\end{theorem}

\section{The groups $\mathbb{Z}_p^2$ with $p\in\{5,7,11,13\}$}\label{sec55}

For $G=\mathbb{Z}_p^2$  we know from \eqref{n=2} that the multiples of $p$ are multiples of $p^{p+3}$. 
As remarked in \eqref{p=3}, for $p=3$ we get all such multiples.
Panraksa recently showed that the same is true for $p=5$, we state that result here and provide a short proof of this in Section~\ref{secPfs4}. 

\begin{theorem}[Panraksa \cite{Panraksa25}]\label{Z55}
$\S(\mathbb{Z}_5^2)= 5^8 \mathbb{Z} \cup \{    m\equiv \pm 1 \text{ or }  \pm 7 \bmod 25\}$.
\end{theorem}

We can also improve \eqref{allp} for a few additional small primes.

\begin{theorem}\label{p>5}
For $p=7$, $11$ or $13$,
\[
p^{p+5}\mathbb{Z} \subseteq \S_0 (\mathbb{Z}_p^2) \subseteq p^{p+3}\mathbb{Z}.
\]
\end{theorem}

This result appears to be new for $p=11$ and $p=13$, though not for $p=7$, as Panraksa \cite{Panraksa49} obtained a complete characterization for $\mathcal{S}_0(\mathbb{Z}_7^2)$.
(In particular, it is established there that the exponents in these inclusions are sharp for $p=7$.)
We include it here since our proof extends to this case naturally.
This theorem immediately implies improved information on $\mathcal{S}_0(\mathbb{Z}_p^n)$ for all $n\geq3$ for these primes, by Theorem~\ref{allmultiples}: for $p=7$, $11$ and $13$ we have that
\[
p^{5+p+2p^2+\cdots+2^{p^{n-1}}}\mathbb{Z} \subseteq \mathcal{S}_0(\mathbb{Z}_p^n).
\]

\section{The group $\mathbb{Z}_3^3$}\label{sec333}

With $w=e^{2\pi i/3}$ we observe that $\mathbb{Z}[w]$ is a unique factorization domain with units $\mathscr{U}=\{\pm 1,\pm w,\pm w^2\}$, with  $3=-w^2(1-w)^2$ and, since $\left(\frac{-3}{p}\right)=\pm 1$ as  $p\equiv \pm 1 \bmod 3$,
the primes $p\equiv1 \bmod 3$ split in $\mathbb{Z}[w]$ while the $p\equiv 2 \bmod 3$ remain prime.
Writing
\[
N(a+bw)=|a+bw|^2=(a+bw)(a+bw^2)=a^2-ab+b^2,
\]
we need to divide the $p\equiv 1 \bmod 9$ according to the form of their factorization.
Let
\begin{equation}\label{eqnT1T2}
\begin{split}
\mathscr{T}_1 & =\{p\equiv 1 \bmod 9 : 3p=N((1-9A)w -(1+9B)),\;  3\nmid A+B\}, \\
\mathscr{T}_2 & =\{p\equiv 1 \bmod 9 : 3p=N((1-9A)w -(1+9B)),\;  3\mid A+B\} .
\end{split}
\end{equation}

We require the set $\mathscr{T}_1$ for our characterization of the integer group determinants of $\mathbb{Z}_3^3$.

\begin{theorem}\label{Z33} Let $G=\mathbb{Z}_3^3$.
\begin{enumerate}[label=(\roman*)]
\item\label{Z331}
The integer group determinants coprime to $3$ are
\be\label{coprime3}  \S^*(G)=\{m : m\equiv \pm 1 \bmod 27\}. \ee
\item\label{Z332}
The group determinants divisible by $3$ are divisible by $3^{18}$.
\item\label{Z333}
All multiples of $3^{20}$ are integer group determinants.
\item\label{Z334}
We achieve determinants of the form  $3^{19} m$, $3\nmid m,$ if and only if $m$ is divisible by a prime $p\equiv 4$ or $7 \bmod 9$, or the square $p^2$ of a prime $p\equiv 2$ or $5 \bmod 9$.
\item\label{Z335}
For the determinants $3^{18}m$, $3\nmid m$,  we achieve all $m$ with $m\equiv \pm 1$ or $\pm 2 \bmod 9$.
We achieve $m\equiv \pm 4 \bmod 9$ if and only if $m$ is divisible by a prime $p\equiv 4$ or $7 \bmod 9$, or a prime $p\in\mathscr{T}_1$, or the square $p^2$ of a prime $p\equiv 2$ or $5 \bmod 9$.
\end{enumerate}
\end{theorem}

We remark that the primes in $\mathscr{T}_1$ less than $5000$ are
\begin{quote}
19, 37, 109, 127, 163, 181, 199, 379, 397, 433, 487, 541, 631, 739, 811, 829, 883, 937, 1063, 1153, 1171, 1279, 1297, 1423, 1459, 1567, 1657, 1693, 1747, 1801, 1873, 1999, 2017, 2053, 2089, 2143, 2161, 2377, 2467, 2503, 2521, 2539, 2557, 2593, 2647, 2683, 2719, 2791, 2917, 2953, 3061, 3169, 3259, 3313, 3331, 3457, 3511, 3547, 3583, 3637, 3673, 3691, 3709, 3727, 3943, 4051, 4159, 4231, 4357, 4447, 4519, 4591, 4663, 4789, 4861, 4987,
\end{quote}
while those in $\mathscr{T}_2$ in the same range are
\begin{quote}
73, 271, 307, 523, 577, 613, 757, 919, 991, 1009, 1117, 1531, 1549, 1621, 1783, 2179, 2251, 2269, 2287, 2341, 2971, 3079, 3187, 3529, 3853, 3889, 3907, 4177, 4339, 4483, 4933, 4951, 4969.
\end{quote}
We remark that the sets $\mathscr{T}_1$ and $\mathscr{T}_2$ also appear in \cite{MP23}, where they are called the Type~1 and Type~2 primes which are $1$ mod $9$.\footnote{These sets are defined slightly differently in \cite{MP23}, as $\mathscr{T}_1=\{p \equiv1\bmod9 : p = N(1+3A(w-1)+Bw, 3\nmid A\}$ and similarly for $\mathscr{T}_2$, but multiplying by $\omega-1$ and applying the norm is easily seen to transform that definition to \eqref{eqnT1T2}.}
For example, it is shown there that
\[
3^4\mathbb{Z} \cap \mathcal{S}_0(\mathbb{Z}_{27}) = \{3^4 p m : p\in\mathscr{T}_1 \textrm{\ and\ } 3\nmid m\}.
\]
Also, from \cite[Lemma 4.9]{MP23}, we have that asymptotically $2/3$ of the primes $p\equiv1\bmod9$ are in $\mathscr{T}_1$ and $1/3$ in $\mathscr{T}_2$.

\section{Proofs of Theorems~\ref{Zpp},~\ref{pp=3} and ~\ref{allmultiples}}\label{secPfs2}

\begin{proof}[Proof of Theorem~\ref{Zpp}]
Suppose that $p\geq 3$ and $G=\mathbb{Z}_p^n$ with $n\geq 2$.
Expanding $F$ in powers of $(1-x_i)$, we write 
$$ F(x_1,\ldots ,x_n)=a_0+ \sum_{i=1}^n a_i(1-x_i) \;\bmod \langle (1-x_i)(1-x_j), \; i,j=1,\ldots ,n\rangle . $$
With $\pi:=1-w$, observe that 
$$1-w^k=(1-w)(1+w+\cdots +w^{k-1})\equiv k\pi  \bmod \pi^2. $$
Suppose that $p\mid M_G(F)$, then $p\mid F(1,\ldots ,1)=a_0$ and $\pi^2\mid a_0$.
Thus
$$ F(w^{j_1},\ldots ,w^{j_n})\equiv  \left( \sum_{i=1}^n a_i j_i  \right) \pi \bmod \pi^2. $$
If $p\mid a_i$ for all $i=1,\ldots ,n$ then $\pi^2\mid F(w^{j_1},\ldots ,w^{j_n})$ for all $j_1,\ldots ,j_n$ giving us an extra $(p^n-1)$ power of $\pi$ and hence $(p^n-1)/(p-1)$ extra $p$ on \eqref{rough}, so exceeding the claimed bounds
in Theorem~\ref{pp=3} and \ref{Zpp}. If $p\nmid a_I$ then for any choice of 
$j_i=0,\ldots ,p-1$ with $i\neq I$, we can solve $j_I\equiv -a_I^{-1}\sum_{i\neq I} a_i j_i \bmod p$ and all those $p^{n-1}$ terms will be
divisible by $\pi^2$; removing the $j_1=\cdots =j_n=0$ term  corresponding to $F(1,\ldots ,1)$,  this gives us an extra $p^{n-1}-1$ powers of $\pi$, and hence $(p^{n-1}-1)/(p-1)$
powers of $p$,  over \eqref{rough}. Hence we get 
\be \label{divrough} p\mid M_G(F) \Rightarrow p^{\frac{p^n-1}{p-1} +\frac{p^{n-1}-1}{p-1}+1}\mid M_G(F). \ee
For $n=2$ or $3$ we have \eqref{n=2} and \eqref{n=3} from \eqref{divrough}.

Suppose now that $p\geq 5$ or $p=3$ and $3^2\mid F(1,\ldots ,1)$  so that $\pi^4\mid F(1,\ldots ,1)$.  
With $G=\mathbb{Z}_p^2$, for a fixed $k\geq0$ take 
$$F_k(x,y)=p^{1+k}+(1-x)+(1-y)^2. $$
Then $F_k(1,1)=p^{1+k}$, and 
$$ F_k(w^s,w^t)\equiv s\pi \bmod \pi^2 $$
for $0\leq t<p$, so that  $\pi \parallel F_k(w^s,w^t)$ except for the terms with $s=0$. For these,
$$ F_k(1,w^t) \equiv t^2\pi^2 \neq 0 \bmod \pi^3 $$
for $t=1,\ldots ,p-1,$  giving us exactly $2(p-1)$ extra $\pi$, so two extra $p$ on \eqref{rough}.

Similarly, for $G=\mathbb{Z}_p^3$ we take
\be \label{extremen=3} F_k(x,y,z)=p^{1+k}+(1-x)+(1-y)^2 -b(1-z)^2 \ee
where $b$ is a quadratic non-residue mod $p$. For $p\geq 5$ or $p=3$ with $k\geq 1$, we have
$$ F_k(w^r,w^s,w^t)\equiv r\pi  \bmod \pi^2 $$
and $\pi \parallel F_k(w^r,w^s,w^t)$ except for the terms $r=0$.
This gains $p^2$ additional factors of $p$ over \eqref{rough}.
For the terms $r=0$, we have
$$ F_k(1,w^s,w^t)\equiv (s^2-bt^2)\pi^2 \bmod \pi^3, $$
where $s^2\not\equiv bt^2 \bmod p$ for $s,t=0,\ldots, p-1$ other than $s=t=0$.
Hence in each of the former $p^2-1$ terms we gain exactly two extra $\pi$, netting $2p+2$ extra powers of $p$, and the $s=t=0$ term contributes exactly $k+1$ additional factors.
The total is an increase of exactly $p^2+2p+3+k$ in the power of $p$ over \eqref{rough}. 

Suppose now that $n\geq 4$. We write
$$ F(x_1,\ldots ,x_n)=a_0+ \sum_{i=1}^n a_i(1-x_i) +\sum_{i,k=1}^n b_{ik}(1-x_i)(1-x_k) +G(x_1,\ldots ,x_n), $$
where $G$ contains the cubic and higher order terms. From
$$  1-w^k = k\pi - \frac{1}{2}k(k-1) \pi^2 \bmod \pi^3, $$
we have
$$ F(w^{j_1},\ldots ,w^{j_n})\equiv  \left( \sum_{i=1}^n a_i j_i  \right) \pi + \left( \sum_{i,k=1}^n b_{ik}j_ij_k -\frac{1}{2}\sum_{i=1}^n a_i  j_i (j_i-1)\right) \pi^2 \bmod \pi^3. $$
As  above we can assume that $p\nmid a_J$, for some $J$, otherwise  $p^{2\cdot\frac{p^n-1}{p-1}+1}\mid M_G(F)$. Suppose that  $p\nmid a_1,$ then we have $\pi^2\mid  F(w^{j_1},\ldots ,w^{j_n})$ for the $p^{n-1}$ values with $j_1\equiv -a_1^{-1}\sum_{i=2}^n a_i j_i$ for $j_2,\ldots ,j_n=0,\ldots ,p-1$. These will have 
$$  F(w^{j_1},\ldots ,w^{j_n})\equiv  Q(j_2,\ldots ,j_n) \pi^2 \bmod \pi^3. $$
where $Q$ is a quadratic in $j_2,\ldots ,j_n$.  We claim that  $Q(j_2,\ldots ,j_n) \equiv 0 \bmod p$ for $p^{n-3}$ choices of 
$j_2,\ldots ,j_n$, giving us  the  $p^{n-3}-1$ extra $\pi$ and so $(p^{n-3}-1)/(p-1)$ extra $p$ as claimed. Suppose first that one of the variables, $j_2$ say, only occurs as a linear term 
$$ Q(j_2,\ldots ,j_n)= j_2L(j_3,\ldots ,j_n)+ Q_2(j_3,\ldots ,j_n),$$
where $L(j_3,\ldots ,j_n)= (\alpha_0+\alpha_3 j_3 +\cdots +\alpha_nj_n)$ with at least one of the $\alpha_i$, say $\alpha_3$,
 not zero mod $p$, and $Q_2$ is quadratic.
Plainly for each of the $p^{n-3}$ choices of $j_4,\ldots ,j_n$ there will be only one $j_3$ with $L_j\equiv 0 \bmod p$. For the remaining 
$p^{n-2}-p^{n-3}$ choices of $j_3,\ldots ,j_n$ we have $L\neq 0$ and we can solve for $j_2$ to give $Q\equiv 0 \bmod p$,
where plainly $p^{n-2}-p^{n-3}>p^{n-3}$.
Hence we can assume that all variables that appear must occur as quadratics, and by completing the square (again ruling out linear terms), we can reduce to the case $Q=\sum_{i=2}^n \alpha_i j_i^2$.
If $Q\equiv 0$ or $\alpha_2 j_2^2$, $p\nmid \alpha_2$, then we have the $p^{n-2}$ solutions with $j_2=0$. So assume that we have at least two $\alpha_i$ with $p\nmid \alpha_i$, say $Q=\alpha_2j_2^2+\alpha_3 j_3^3 + Q_3(j_4,\ldots ,j_n)$. Now for any of the $p^{n-3}$ choices of $j_4,\ldots ,j_n,$ the $\alpha_2j_2^2$ and $-\alpha_3 j_3^2-Q_3$
each take on $(p+1)/2$ values mod $p$, so by the box principle there is a choice of $j_2,j_3$ making $Q\equiv 0 \bmod p$.

For sharpness when $n=4$, we take $b$ a quadratic non-residue mod $p$ and
\be \label{extremen=4}  F_k(x,y,z,t)=p^{1+k}+(x-1)+(y-1)^2-b(z-1)^2+(t-1)^3. \ee
If $p\geq 5$ or $p=3$ and $k\geq 1$, we get $\pi\parallel F_k(w^{i_1},w^{i_2},w^{i_3},w^{i_4})$ unless $i_1=0$, and $\pi^2\parallel F_k(1,w^{i_2},w^{i_3},w^{i_4})$ unless $i_2=i_3=0$, and $\pi^3\parallel F_k(1,1,1,w^{i_4})$ unless $i_4=0$. 

For $n=5$, select a polynomial $x^3 + \sum_{i=0}^2 A_ix^i$ in $\mathbb{Z}[x]$ which is irreducible mod $p$.
For example, we may choose the minimal polynomial for a generator of $\mathbb F_{p^3}$ over $\mathbb{Z}_p$, or when $p\equiv 1 \bmod 3$, we can take $x^3-c$ where $c$ is not a cubic residue mod $p$.
Then set
$$ F_k(x,y,z,t,s)=p^{1+k}+(x-1)+(y-1)^2-b(z-1)^2+(t-1)^3+\sum_{i=0}^2  A_i (t-1)^i(s-1)^{3-i}. $$
We have $\pi \parallel F_k(w^{i_1},w^{i_2},w^{i_3},w^{i_4},w^{i_5})$ unless $i_1=0$, $\pi^2\parallel F_k(1,w^{i_2},w^{i_3},w^{i_4},w^{i_5})$ unless $i_2=i_3=0$, and $\pi^3\parallel F_k(1,1,1,w^{i_4},w^{i_5})$ unless $i_4=i_5=0$.
\end{proof}

\begin{proof}[Proof of Theorem~\ref{pp=3}] 
For $p=3$, we get $3^{2\cdot 3^{n-1}}\mid M_G(F)$ from \eqref{divrough}, where
$$M_G(2+x_1^2)=M_{\mathbb{Z}_3}\left((1-x)+(1+x+x^2)\right)^{p^{n-1}}= 3^{2\cdot 3^{n-1}}. $$
For $n=3$ and $k\geq 1$, we can get $3^{18+k}\parallel M_G(F)$ from \eqref{extremen=3}, and for $n=4$ and $k\geq 1$, we can get $3^{55+k}\parallel M_G(F)$ from \eqref{extremen=4}, with $b=-1$.
For $n=5$ and $k\geq 1$ we can get $3^{166+k}\parallel M_G(F)$ from
$$ F(x,y,z,t,s)=p^{1+k}+(x-1)+(y-1)^2+(z-1)^2+(t-1)^3+\Phi_3(t)(s-1). $$
If $n\geq 4$ and  $3^2\mid F(1,\ldots ,1)$, we showed in the proof of \eqref{n>3} that $p^a\mid M_G(F)$ with 
$$a\geq \frac{p^n-1}{p-1}+ \frac{p^{n-1}-1}{p-1} + \frac{p^{n-3}-1}{p-1} + 2 = 2\cdot 3^{n-1}+\frac{1}{2}(3^{n-3}+1).$$
 If $F(1,\ldots ,1)=3m$ with $3\nmid m$ then the same argument replaces $Q$ by $Q+m$,
and as long as $Q \bmod p$ contains at least two variables the same argument gives at least $p^{n-3}$ zeros and, since this doesn't include $(0,\ldots ,0),$ we get at least $\frac{1}{2}(3^{n-3}+1)$ extra $3$'s over $2\cdot 3^{n-1}$. If $Q=0 \bmod 3$ then $Q+m\not\equiv 0 \bmod p$ and $3^{2\cdot 3^{n-1}}\parallel  M_G(F)$. This leaves the case $m+aj_2^2 \bmod p$. If $-ma^{-1}$ is a quadratic non-residue mod $p$ then again we don't represent zero and  $3^{2\cdot 3^{n-1}}\parallel  M_G(F)$. If it is a quadratic residue then there will be two solutions $k_2$
and $p^{n-2}$ choices for $j_3,\ldots ,j_n$ adding $2\cdot 3^{n-2}$ extra $u$'s and hence $p^a\mid M_G(F)$ with $a\geq 2\cdot 3^{n-1}+ 3^{n-2}$.
\end{proof}

\begin{proof}[Proof of Theorem~\ref{allmultiples}] For $H$ in $\mathbb{Z}[x_2,\ldots ,x_n]$
we have 
\[
M_{\mathbb{Z}_p^n}(x_1-1 +\Phi_p(x_1)H)=p^{2p^{n-1}}M_{\mathbb{Z}_p^{n-1}}(H).
\]
The results for $\mathbb{Z}_p^n$ then follow recursively starting with $\mathbb{Z}_p$:
\[
M_{\mathbb{Z}_p}(x-1+t\Phi_p(x))=tp^2, \; t\in \mathbb{Z}, \quad M_{\mathbb{Z}_p}(\pm (1+x+\cdots + x^{m-1}))=\pm m, \;p\nmid m,
\]
or from $\mathbb{Z}_p^{n-1}$ using \eqref{coprimevalues}.
\end{proof}

\section{Proofs of  Theorems~\ref{withpower}, \ref{genallp}, \ref{twocomponents} and~\ref{thmSstarZ3Z9}}\label{secPfs3}

\begin{proof}[Proof of Theorem~\ref{withpower}]
For $j\geq 1$ we write
\be \label{defpij}\omega_j=e^{2\pi i/p^j},\;\;  \pi_j = \omega_{j}-1,\quad \abs{\pi_j}=p^{-1/\phi(p^j)},\quad \pi_{j+1}^p \mid \pi_j. \ee
In addition to $\mathbb{Z}_p \times \mathbb{Z}_p^{n-1}$, for each $2\leq j\leq t$ we have $\phi(p^j) p^{n-1}$ terms
$$ \pi_j \mid F\left(\omega_{j}^{j_1},\omega^{j_2},\ldots ,\omega^{j_n}\right), \;\; 1\leq j_1<p^j, \; \gcd(j_1,p)=1,\;\; 0\leq j_2,\ldots ,j_n<p, $$
contributing $p^{p^{n-1}}$. For the examples showing sharpness for $\mathbb{Z}_p^n$ it is easily seen that $\pi_j \parallel F(x_1,\ldots ,x_n)$ for these extra terms.
\end{proof}

\begin{proof}[Proof of Theorem~\ref{genallp}]
The upper inclusion in \ref{genallp1}, and the sharpness of the exponent here, is from Theorem \ref{withpower}, the lower 
from
\[
M_G\left( x-1+ \prod_{j=1}^t \Phi_{p^j}(x)\left(y-1 +m \Phi_p(y)\right)\right)= mp^{2tp+2}.
\]
In addition,
\[
M_G\left(x-1+\left(\frac{x^{p^t}-1}{x-1}\right)\right)= p^{2tp}.
\]
For \ref{genallp2}, suppose that  $M_G(F)=p^{pt+3}m,$ $p\nmid m,$  with $t\geq 3$ or $t=2$ with $p\geq 5,$ and write
$$ F(x,y)=A+B(x-1)+C(y-1) +\sum_{i+j\geq 2} a_{ij}(x-1)^i(y-1)^j. $$
When $p\nmid B$ then it is readily seen from the proof of Theorems~\ref{Zpp} and~\ref{withpower} that $M_{\mathbb{Z}_{p^t}}(F(x,1))=p^{t+1}m_1,$ for some $m_1\mid m$, so $m=1$ is not possible by \eqref{eqnNewmanRestric}.
If $p\mid B$ then $\pi_i^2\mid F(\omega_i,y),$ $i=2,\ldots ,t$, giving an additional $p^{(t-1)p}$.

For \ref{genallp3}, when $G=\mathbb{Z}_9 \times \mathbb{Z}_3$ we have $M_G(F)=3^{11}(3m\pm 1)$ from 
$$ m\left( \frac{x^9-1}{x-1}\right) \left( \frac{y^3-1}{y-1}\right) \pm \left(x^7-1 +y(1+x+x^2+x^3+x^4+x^8)+y^2(1+x+x^2)\right) $$
and $3^{12}m$ from
$$ F(x,y)=m\left( \frac{x^9-1}{x-1}\right) \left( \frac{y^3-1}{y-1}\right) -(1 +x + x^8) +y(x^2-1) + y^2(x^3+x^2+x).$$
Finally, for \ref{genallp4}, the lower bound in \eqref{lower3s} follows from recursively applying
\be \label{iterate} G=\mathbb{Z}_{p^t}\times G',\;\; M_G\left(x_1-1 + \left(\frac{x_1^{p^t}-1}{x-1} \right)f(x_2,\ldots ,x_n)\right)=p^{2t|G'|}M_{G'}(f) \ee
starting with $M_{\mathbb{Z}_3^2}(f)=m3^6$ to produce $m3^{6+2\cdot 3^2+2\cdot 3^3+\cdots + 2\cdot 3^{n-2}+4\cdot  3^{n-1}}$.
\end{proof}

\begin{proof}[Proof of Theorem~\ref{twocomponents}]
With $\omega_j$ and $\pi_j$ as in \eqref{defpij} we note that
\be \label{omegacong}  \omega_{j}^k -1 \equiv k\pi_j + \frac{1}{2}k(k-1)\pi_j^2 \text{ mod } \pi_j^3. \ee
Suppose that 
$$ F(x,y)=A+B(x-1)+C(y-1)+\sum_{u+v\geq 2} a_{uv}(x-1)^u(y-1)^v, $$
with $p\mid M_G(F).$
Then $p\mid A$ and for $x$ and $y$ primitive $p^j$th and $p^i$th roots of unity respectively, we have $\pi_{\max \{i,j\} } \mid F(x,y).$ 
For each pair $i,j$, $0\leq j \leq t$, $0\leq i\leq s,$  these terms therefore contribute at least $\phi(p^{\min\{i,j\}})$ many $p$'s to $M_G(F)$.
For $j\leq s$ we have
$$ \sum_{i=0}^s  \phi \left( p^{\min\{i,j\}}\right) = \sum_{0\leq i\leq  j} \phi (p^i) + \sum_{i=j+1}^s \phi(p^j)=p^j+ (s-j)\phi(p^j).    $$
For $j>s$ we have
$$ \sum_{i=0}^s  \phi \left( p^{\min\{i,j\}}\right) = \sum_{0\leq i\leq  s} \phi (p^i) =p^s.    $$
Hence in total we have
\begin{align*}  \sum_{j=0}^s \left(p^j + (s-j)\phi(p^j)\right) + \sum_{j=s+1}^t p^s  & = \sum_{j=0}^s p^j+ \sum_{j=0}^{s-1} p^j+  (t-s)p^s \\
& =p^s+2\left(\frac{p^s-1}{p-1}\right) +(t-s)p^s. 
\end{align*}

Suppose first that $p\nmid BC$. For the $i\neq j$ plainly $\pi_{\max \{i,j\} } \parallel  F(x,y).$ 
For the terms with $i=j$, $1\leq j\leq s$, we have by \eqref{omegacong}
$$ F( \omega_j^u,\omega_j^v)\equiv (Bu+Cv)\pi_j \text{ mod } \pi_j^2,\hspace{3ex} 1\leq u,v<p^j, \gcd(uv,p)=1. $$
Hence we have $\pi_j^2 \mid F( \omega_j^u,\omega_j^v)$ for the $\phi(p^j)p^{j-1}$ 
values $u,v$ coprime to $p$ with $u\equiv -CB^{-1}v \bmod p$ and $\pi_j\parallel F( \omega_j^u,\omega_j^v)$
for the others. These contribute an extra $\sum_{j=1}^s p^{j-1}$ many $p$'s and the exponent $\alpha$ claimed.
Notice that if we take 
$$F(x,y)=p^{k+1}-(x-1)+c(y-1), \quad \left(\frac{c}{p}\right)=-1,$$
then \eqref{omegacong} gives   (setting $c=-1$ and checking separately for $j=1$ if $p=3$, $k=0$)
$$ F( \omega_j^u,\omega_j^v)\equiv (cv-u)\pi_j + \frac{1}{2}\left(cv(v-1)-u(u-1)\right) \pi_j^2\;\;   \text{ mod } \pi_j^3, $$
and for the $u\equiv cv \bmod p$
$$ F( \omega_j^u,\omega_j^v)\equiv 2^{-1}(cv^2-u^2)\pi_j^2 \not\equiv 0 \text{ mod } \pi_j^3, $$
and  $p^{\alpha+k}\parallel M_G(F)$.

If $p\mid B$ then $\pi_j^2\mid F(x, y)$  if $x$ is a primitive $p^j$th root of unity and $y$ is any $p^{(j-1)}$th root of unity. Thus we get an extra $\sum_{j=1}^s p^{j-1} + \sum_{j=s+1}^t p^s$ contribution of $p$'s.

Likewise, if $p\mid C$ then $\pi_i^2 \mid F(x,y)$ for a primitive $p^i$th root of unity $y$ and $x$ any $p^{i-1}$th root of unity, giving again an extra $p^{i-1}$ many $p$'s for each $1\leq i\leq s$. 
In this case we can also get equality by taking 
$F(x,y)=p^{1+k}+(x-1)-(y-1)^2$; it is readily checked that $\pi_i^2\parallel F(x,y)$ in these extra cases and 
$\pi_{\max\{i,j\}}\parallel F(x,y)$ in the others.
To obtain in  the lower bound \eqref{lowerks}, use \eqref{iterate} with 
$M_{\mathbb{Z}_{p^s}}\left(1-y+ m\left(\frac{y^{p^s}-1}{y-1}\right)\right)=mp^{2s},$ while
$M_G\left(1-x+ m\left(\frac{x^{p^t}-1}{x-1}\right)\right)=p^{2kp^s}.$
\end{proof}

\begin{proof}[Proof of Theorem~\ref{thmSstarZ3Z9}]
Write 
\[
F(x,y)=f(y)+(x-1)g(y)+(x-1)^2h(y),
\]
so that the corresponding $\mathbb{Z}_3\times \mathbb Z_9$ determinant is 
\[
D=\prod_{x^3=1}\prod_{y^9=1} F(x,y)=A\abs{B}^2,\quad A=\prod_{y^9=1}f(y),\;\; B=\prod_{y^9=1}F(\om ,y).
\]
From \cite[Theorem~2.1]{pgroups}, we have $\abs{B}^2\equiv A^2\bmod 27$, so that a determinant $D$ with $3\nmid D$ necessarily satisfies 
\[
D\equiv A^3\equiv \pm 1,\pm 8, \pm 10 \bmod 27.
\]

We show first that each of the values described in the statement is realized as a determinant.
To achieve the $D=\pm 1 + 27m$, we take
\[
F(x,y)=\pm 1 + m(x^2+x+1)\frac{y^9-1}{y-1},
\]
and to achieve $D=A N( A+9(1-\om)u -27\om v)$, we assume $A>0$ (otherwise multiply by $-1$) and take
\[
F(x,y) = \frac{y^A-1}{y-1}+u (1-x)\frac{y^9-1}{y-1} + v(1-x)^2\frac{y^9-1}{y-1} \bmod (y^9-1).
\]
This has $F(1,1)=A$, $F(\om,1) = A + (1-\om)9u - 27\om v$, and $F(\om^2,1)$ its conjugate.
The remaining $F(x,y)$ are $(y^A-1)/(y-1)$ and have resultant $1$ with the third and ninth roots of unity.
We thus obtain \eqref{nice} with $B=A+9u$ and $C=-9u-27v$.

Next, we show that all determinants $D\equiv\pm8$ of $\pm10\bmod27$ have the form \eqref{nice}.
Write
\[
\Gamma := \prod_{y^9=1}\left( f(y)+xg(y)+x^2h(y) \right)= A + \sum_{j=1}^{18} A_j x^j,
\]
where the $A_j$ are integer polynomials in the $27$ coefficients of $f(y)$, $g(y)$ and $h(y)$.
We claim that $A_1,A_2\equiv 0\bmod 9$ and $A_3,A_4\equiv 0\bmod 3$. 
Since $|y-1|_3<1$ for the ninth roots of unity, we have $\Gamma\equiv (f(1)+xg(1)+x^2h(1))^9 \bmod 3$, and so plainly $A_1,\ldots ,A_8 \equiv 0\bmod 3$.
To see that $A_1,A_2\equiv 0 \bmod 9$, we calculate 
\[
\prod_{y^9=1} \left( f(y)+xg(y)  \right)
\]
and verify that the coefficients of $x$ and $x^2$ (integer polynomials in the 18 coefficients of $f(y)$ and $g(y)$) are both multiples of $9$.
We do not need the $x^2h(y)$ term since the $x^2$ terms stemming from any $x^2h(y)$ term will be the same as the $x$ term arising in the product of $f(y)+xh(y)$ and so is also a multiple of $9$.
Observing that $(1-\om)^2=-3\om$, we obtain
\begin{align*}
B & = A+9\alpha_1 (1-\om)+ 9\alpha_2(1-\om)^2+3\alpha_3 (1-\om)^3\\
&\qquad +3\alpha_4(1-\om)^4+(1-\om)^5(\alpha_5+\alpha_6\om) \\
 & = A + (1-\om)(9\beta_1+9\beta_2 \om)+ 27(\beta_3+\beta_4\om)\\
 & = A + 9(1-\om)(\beta_1+\beta_2+3\beta_3)+27\om (\beta_2+\beta_3+\beta_4).
\end{align*}
Thus $D=A|B|^2$ has the form \eqref{nice} as required.
\end{proof}

\section{Proofs of Theorems~\ref{Z55} and~\ref{p>5}}\label{secPfs4}

Suppose that $G=\mathbb{Z}_5^2$. From \eqref{coprimevalues} we know that the values coprime to $5$ are the $x^4\equiv 1$  mod $25$, that is $x\equiv \pm 1,\pm 7 \bmod 25$.
With $g_p=\Phi_p(x)\Phi_p(y)$, we get the $5^9m$ from
$$ (y^4-1)+x(1-y^2+y^4)-x^2(y^2+y^3)+x^3(1+y-y^4)+x^4(y^2-y) +mg_5.    $$
We get the $5^8(1+5m)$ from
$$ (y^4-1)+x(1-y^2+y^4)+x^2(1+y+y^4)+x^3(1+y-y^4)+x^4(y^2-y) +mg_5,    $$
the $5^8(2+5m)$ from 
$$ (1+y^2+y^3) +x(1+y+y^2+y^3)+x^2(y+y^2+y^4)+mg_5, $$
and  the $-5^8(1+5m),-5^8(2+5m)$ from their negations. 

\pushQED{\qed}
For $G=\mathbb{Z}_p^2$ with $p=7$, $11$, $13$ we have
\begin{align*}
&M_{\mathbb{Z}_7^2}\left(1 + (x-1) (x^3-1) (y-1) - x^4 y+mg_7\right) = 7^{12}m,\\
&M_{\mathbb{Z}_{11}^2}\left(1 + (x-1) (x^3-1) y(y-1) - x^4 y^3+mg_{11}\right)=11^{16}m,\\
&M_{\mathbb{Z}_{13}^2}\left(1 + (x-1) (x^5-1) (y-1) - x^6 y+mg_{13}\right)=13^{18}m.\qedhere
\end{align*}
\popQED
While these examples possess some evident structure, we remark that some searches over polynomials with similar forms did not uncover examples for larger primes.

\section{Proof of Theorem~\ref{Z33}}\label{secPfs5}

We require the following lemma.
\begin{lemma}\label{norm}
Let $p$ be a prime number.
\begin{enumerate}[label=(\roman*)]
\item\label{norm1}
If $p=1 \bmod 3$ we can write
\be \label{form1}  p=N( (3a-1)w+(3b-1)) \ee
with $a+b\equiv 1$, $-1$ or $0 \bmod 3$  as $p\equiv 7$, $4$ or $1 \bmod 9$.
Alternatively, we can write
\be \label{form2}  p=N( 2-3a-3bw) \ee
with $a+b\equiv -1$, $0$ or $1 \bmod 3$  as $p\equiv 7$, $4$ or $1 \bmod 9$.
\item\label{norm2}
If $p\equiv 1 \bmod 9$ then we can write
\be \label{p=1}  3p=N((1-9A)w-(1+9B)) .\ee
\item\label{norm3}
If  $p\equiv 4 \bmod 9$ then we can write
\be \label{p=4}  3p=N((2+9A)w-(2-9B)), \ee
with $A+B \equiv 0 \bmod 3$ or with either  $A+B\equiv \pm 1 \bmod 3$.
\item\label{norm4}
If $p\equiv 7 \bmod 9$ then we can write
\be \label{p=7}  3p=N((1+9A)w+(5+9B)), \ee
with $A+B\equiv 1 \bmod 3$ or with either $A+B\equiv -1$ or $0 \bmod 3$.
\end{enumerate}
\end{lemma}

\begin{proof} Since $p\equiv 1 \bmod 3$ factors in $\mathbb{Z}[w]$, we can write
$$ p=N(aw+b)=a^2-ab+b^2. $$
Plainly we can not have both $3\mid a$ and $3\mid b$, or $b\equiv -a \bmod 3$, else $N(a+bw)\equiv 0 \bmod 3$ and $p\neq 3$. 
Since $N(aw+b)=N(a+bw)=N(-aw-b)$, we can assume that $3\nmid a$, then $a\equiv -1 \bmod 3$. If $3\mid b$ we
have $N(aw+b)=N(-aw^2-b)=N(aw+(a-b))$, so we can assume that $a\equiv b\equiv -1 \bmod 3,$ giving us form \eqref{form1}.
From 
$$ N((3a-1)w+(3b-1))=1-3(a+b)+9(a^2-ab+b^2) $$
we have $a+b\equiv 1,-1$ or $0 \bmod 3$ as $p\equiv 7,4$ or $1 \bmod 9$ respectively.
Writing
$$ -w \left( (3a-1)w + (3b-1)  \right) = 2-3(1-a) - 3(b-a)w $$
produces \eqref{form2}.
If $p=N( (3a-1)w+(3b-1))$ then $3p$ is the norm of 
$$ (1-w^2)( (3a-1)w^2+(3b-1))=( 1-9a+3(a+b))w- (1+9a-6(a+b)).   $$
When $3\mid a+b $ we get form \eqref{p=1}, and when $a+b\equiv -1 \bmod 3$ we get, from its negative, \eqref{p=4}.
In the latter case, $N((2+9A)w-(2-9B))=N((2-9B)w-(2+9A))$, and a representation with $A+B\equiv 1 \bmod 3$ gives one with $A+B\equiv -1 \bmod 3$ and vice versa.

When $a+b\equiv 1 \bmod 3$ we get 
$$ 3p=N( (4-9s)w+(5-9t)), $$
and
$$ (4-9s)w^2+(5-9t)w= 5+9(s-1)+ (1+9(s-t))w $$
yields the form \eqref{p=7}.  Since 
$$ -(1+9A)w^2-(5+9B) = (1+9A)w + (5+9(A-B-1)) $$
has the same norm as $(1+9A)w+(5+9B)$, a representation with $A+B\equiv -1 \bmod 3$ will give one with $A+B\equiv 0 \bmod 3$ and vice versa.
\end{proof}

Since
\[
(1-9A)w-(1+9B) =-(1-w)\left(   1+3(1-w) \left( Bw + (B-A)\right) \right)
\]
we can equivalently write 
\begin{align*}
\mathscr{T}_1 & =\{p\equiv 1 \bmod 9 : p=N(1+3(1-w)(aw+b)),\;  3\nmid a+b\} \\
\mathscr{T}_2 & =\{p\equiv 1 \bmod 9 : p=N(1+3(1-w)(aw+b)),\;  3\mid a+b\}.
\end{align*}
Notice that if $p\equiv -1 \bmod 9$ then 
$$ -p=1+9t= 1+3(1-w)(tw +2t) $$
is also of the form $1+3(1-w)(aw+b)$, $3\mid (a+b)$.
Moreover, observe that the conjugate of a number of this form is of this form:
$$   1+3(1-w^2)(aw^2+b) = 1+3(1-w)((b-a)w+b), $$
as is the product of two numbers of this form:
\begin{equation*}
\begin{split}
&\left(1+3(1-w)(a_1w+b_1) \right)\left( 1+3(1-w)(a_2w+b_2)\right) =\\
&\qquad 1+3(1-w)( (a_1+a_2+3\kappa) w+(b_1+b_2+3\lambda) ).
\end{split}
\end{equation*}
Also, $-(1-w)$ times a number of this form  will be of the form $(1-9A)w-(1+9B)$ with $3\mid A+B$.

\begin{proof}[Proof of Theorem~\ref{Z33}]
The values \eqref{coprime3} coprime to $3$ in \ref{Z331} follow from \eqref{3coprime}, and \ref{Z332} is from Theorem~\ref{pp=3}. 
Set
$$ g:=\Phi_3(x)\Phi_3(y)\Phi_3(z). $$
For \ref{Z333}, the multiples $3^{20}$ are achieved with
$$ M\left( 1+2x-x\Phi_3(y)+\Phi_3(x)(z-1)+ m g\right)=3^{21}m $$
and
$$ M\left( 1+2x-x\Phi_3(y)+\Phi_3(x)(z-1)+ \Phi_3(x)\Phi_3(y)+mg \right)=3^{20}(1+3m) $$
or its negative. 

For \ref{Z334}, we achieve the $3^{18}m$ with $m\equiv \pm 1$ or $\pm 2 \bmod 9$ using
\begin{gather*}
 M\left( 1-x + \Phi_3(x)+mg \right) = 3^{18}(1+9m), \\
 M\left( x\Phi_3(y)-(1+2x)+(1+z)\Phi(x)+mg \right) = 3^{18}(2+9m), 
\end{gather*}
or their negatives.

Suppose now that we want $3^{18}m$ with  $m\equiv 5 \bmod 9$, replacing $F$ by $-F$ to get the $m\equiv 4 \bmod 9$.
Suppose first that $m=km'$ with $k \equiv 7 \bmod 9$, $m'\equiv 2 \bmod 9$  where $k=p$  is a prime  with $p\equiv 7 \bmod 9$ or  $k=p^2$ the square of a prime  $p\equiv 5 \bmod 9$
so that we can write 
$$ k=N((3A-1)w+(3B-1)),\;\; A+B\equiv 1 \bmod 3, $$
from \eqref{form1} in Lemma~\ref{norm} when $k=p$ and from $p^2=N\left((3(3t+2)-1)w+3(3t+2)-1\right)$ when $p=9t+5$,
 and achieve $3^{18}km'$ with
$$ F=(Ay+B) \Phi(z)\Phi(x) -\Phi(x)+ (x\Phi(y)-y(2x+1)) +m_1 g, \quad m_1=\frac{1}{9}(m'+1-3(A+B)). $$

Suppose next that $m=km'$ with $k \equiv 4 \bmod 9$, $m'\equiv 1 \bmod 9$  where $k=p$  is a prime  with $p\equiv 4 \bmod 9$ or  $k=p^2$ the square of a prime  $p\equiv 2 \bmod 9$
so that we can write
$$ k=N((2-3B)-3Aw),\quad 3\mid A+B, $$
from \eqref{form2} when $k=p$ and $p^2=N(2-3(3t))$ when  $p=2-9t$, and 
$$ F=\Phi(x)-y( x\Phi(z)-(1+2x) ) -\Phi(x)\Phi(y)(Az+B) +m_1g,\quad m_1=\frac{1}{9}(m'-1+3(A+B)). $$
produces $3^{18}m'k.$

Suppose that  $m=pm'$ with $p\equiv 1 \bmod 9$  in $\mathscr{T}_1$  and $m'\equiv 4 \bmod 9$. Then by \eqref{p=1} in Lemma~\ref{norm} we can write
$$ 3p=N( (1-9A)w-(1+9B)),   \hspace{3ex} 3\nmid A+B, $$
and we can achieve $3^{18}pm'$ with 
$$ F=E(x-1) +\Phi(x)-E(Ax+B )\Phi(z)\Phi(y)+m_1g, $$
where
$$ E:=\begin{cases} -1 & \hbox{ if $A+B\equiv 1 \bmod 3$},\\ 1 & \hbox{ if $A+B\equiv -1 \bmod 3$},\end{cases} \hspace{3ex} m_1:=\frac{1}{9}(m'-1+3E(A+B)). $$

Suppose now that we try to make a determinant of the form $3^{18}m$, $m\equiv 4 \bmod 9$ where $m$ contains no prime $p\equiv 4$ or $7 \bmod 9$ or from $\mathscr{T}_1$ or the square or a prime $2$ or $5 \bmod 9$.
Notice to obtain $3^{18}\parallel M$ we must have $3\parallel F(1,1,1)$, $3\parallel  F(w^i,w^j,w^k)$ for eight other terms, and $(1-w)\parallel F(w^i,w^j,w^k)$ for the remaining 18 terms.
Supposing that the expansion of $F$ has linear terms $a(x-1)+b(y-1)+c(z-1)$, we know that we cannot have $3\mid a,b,c$ (else $3$ divides all terms), and by exchanging a pair of variables if necessary we can assume $3\nmid a$.
With the change of variables $x\mapsto x y^{-ba^{-1}}z^{-ca^{-1}}$ and reducing exponents  mod $3$, we can assume that $3\mid b,c$ and $3\parallel F(1,w^j,w^k)$.
Since a prime $p\equiv 2 \bmod 3$ does not split in $\mathbb{Z}[w]$, and any $p\mid F(w^i,w^j,w^k)$  must also divide its conjugate, and  $m$ is not divisible by $p^2$ with $p \bmod 2$ or $5 \bmod 9$, we see that any primes in one of these two residue classes mod $9$ must divide $F(1,1,1)$. Hence, apart from the powers of $1-w$, the $F(w^i,w^j,w^k)$ other than $F(1,1,1)$ can only be divisible by products of primes $p=\pm 1 \bmod 9$
or factors of primes  $p\equiv 1 \bmod 9$ in $\mathscr{T}_2$. As discussed above this makes
\begin{align*} F(1,1,1) & =\pm 3(4+9\alpha),\\
 F(1,w^j,w^k) & = 3(1+3(1-w)(\alpha w +\beta)) u, \;\; 3\mid \alpha +\beta, \\
 F(w^i,w^j,w^k) & = ((1-9\alpha)-(1+9\beta )w) u, \;\; 3\mid \alpha+\beta, 
\end{align*}
where the $u$ are units. Observe that if $F(x,y,z)=\sum_{i,j,k=0,1,2} a_{ijk}x^iy^jz^k$ then the coefficients can be recovered
from the values $F(w^i,w^j,w^k)$:
\be \label{RecoverCoeff}  27 a_{IJK}= \sum_{i,j,k=0,1,2} w^{-iI-jJ-kK} F(w^i,w^j,w^k), \ee
and will correspond to an integer polynomial when the sums on the right are all multiples of 27. 
Notice that for $3\mid \alpha + \beta$ we have
$$  9(1-w)(\alpha w+\beta)= 27(1-w)\frac{(\alpha+\beta)}{3} + 27w\alpha,  $$
while
$$-(9\alpha+9\beta w)u =-27 \frac{(\alpha+\beta)}{3} u +9\beta(1-w)u,  $$
where $(1-w)u+(1-w^2)\bar{u}=\pm 3,0,\mp 3$ as $u=\pm 1,\pm w,\pm w^2$, so that when adding $w^{-iI-jJ-kK}F(w^i,w^j,w^k)$
and its conjugate the terms involving $\alpha$, $\beta$ contribute a multiple of 27.  Hence the existence of a solution 
depends on whether there is a solution when all the $\alpha$, $\beta$ are zero.  That is, replacing $F$ by $-F$ as necessary, 
we need only consider
$$ F(1,1,1)=12, \quad  F(1,w,w^k)=3u_{k},\; k=0,1,2, \quad F(1,1,w)=3u, $$
$$ F(w,w^j,w^k)=(1-w)u_{jk},\;\;  j,k=0,1,2,  $$
with units $u$, $u_{k}$ and $u_{jk}$, the other values coming from the conjugates. Moreover, dividing the units as to whether $u\equiv \pm 1 \bmod 1-w$, 
$$ \mathscr{U}_1:= \{1,w,w^2\},\quad       \mathscr{U}_2:= \{-1,-w,-w^2\},  $$
and observing that $\sum_{j=0,1,2}   F(w,w^j,w^k)$ and  $\sum_{k=0,1,2}   F(w,w^j,w^k) \equiv 0 \bmod 3$, it is readily seen that the $u_{jk}$ are all in $\mathscr{U}_1$ or all in $\mathscr{U}_2$. Similarly, since 
$$\sum_{j,k=0,1,2}F(1,w^j,w^k)\equiv 0 \bmod 9, $$
 we see that if $t$ of the units $u,u_0,u_1,u_2$ are in $\mathscr{U}_1$ and $(4-t)$ are in $\mathscr{U}_2$, then $4+2t-2(4-t)\equiv 0 \bmod 3$, so that $t=1$ or $4$. Replacing $F$ by $x^a y^b z^c F \bmod \langle x^3-1, y^3-1, z^3-1 \rangle$, 
we can also assume that $u_{00}$, $u_0$, $u$ are $\pm 1$. That is we have only $2\cdot (1 + 4)\cdot 3^{10}=590,490$ different choices of the $13$ units.
An exhaustive search of all possible choices shows that no solution to $3^{18}\cdot 4$ exists with $F(1,1,1)=12$.
It follows that the only determinants we can achieve of the form $3^{18}m$ with $3\nmid m$ are those stated in the theorem.

Consider now the $3^{19}m$ with $3\nmid m$ for \ref{Z335}.
Suppose first that $m=km'$ where $ k\equiv 7 \bmod 9$ is a prime $p\equiv 7 \bmod 9$, or the square of a prime $p\equiv 5 \bmod 9$. So we may write
$$ 3k=N((5+9B) + (1+9A)w),   \;\; 1+ A+B\equiv \pm 1  \bmod 3, $$
from \eqref{p=7} when $k=p$ and from $3p^2=N((1+2w)p)=N(5+9t+ (1+9(1+2t))w)$ when $p=5+9t$, 
and one can get $3^{19} km'$
with
$$F=\lambda \Big(  \Phi(x)+y\Phi(z)+z\Phi(y) +zy(x-1) + (Ax+B) \Phi(y)\Phi(z)\Big)+ m_1 g$$
with $\lambda=\pm 1$ chosen so that $\lambda (1+A+B)\equiv m' \bmod 3$ and $m_1=(m'-\lambda (1+A+B))/3$.

Suppose next that $m=m'k$ with $k\equiv 4 \bmod 9$, and $k$ is either a prime $p\equiv 4 \bmod 9$ or the square $p^2$ of a prime $p\equiv 2 \bmod 9$. Then we can write 
$$ 3k=N((9A+2)w+(9B-2)    ), $$
with $A+B\equiv 0 \bmod 3$ or with both $A+B\equiv \pm 1 \bmod 3$, from \eqref{p=4} in Lemma~\ref{norm} when $k=p$, and trivially $3p^2=N( (9A+2)w+(-9A-2))$ when $p=9A+2$.
With
$$F= (1+yz)(\Phi(x)+yx)-(1+z)
+(Ax+ B )\Phi(y)\Phi(z)+m_1 g$$
we get $3^{19}nk$ with
$$ n=2+3A+3B+9m_1. $$
From
\begin{equation*}
\begin{split}
F &= yz^2-xy^2z+y^2+x^2yz^2-x+x^2yz-xy+x^2y^2z^2+x^2y^2\\
&\qquad + ((A+1)x+B )\Phi(y)\Phi(z) + m_1 g
\end{split}
\end{equation*}
we get $3^{19}nk$ with
\[
n=4+3A+3B+9m_1,
\]
and from 
\[
F=1+x+z(y+x+xy)-y(1+x^2) +((A-B)x- B ) \Phi(y)\Phi(z) + (m_1+B) g,
\]
and observing that  $(4+9A-9B)w+(2-9B)=(2+9A)w+(9B-2)w^2$ and so
\[
N( (4+9A-9B)w+(2-9B) )=3k,
\]
we get
$3^{19}kn$ with 
$$
n=1+3A+3B+9m_1.
$$
If $A+B\equiv 0 \bmod 3$ then as we vary $m_1$ these three polynomials produce all  the integers $n\equiv 1,2,4 \bmod 9$, and with $-F$ all the $n\equiv 8,7,5 \bmod 9$, and we achieve all $n$ coprime to $3$.
In the remaining cases we have both $A+B\equiv1$ and $ -1 \bmod 3$ and so run through the $n\equiv 5$ and $8$, $7$ and $1$,  $4$ and $7 \bmod 9$, with $-F$ adding in the missing $n\equiv 2 \bmod 9$, again  achieving all possible $3\nmid n$.

Suppose now that we try to make a determinant of the form $3^{19}m$,  $3\nmid m$, where $m$ contains no prime $p\equiv 4$ or $7 \bmod 9$  or the square of a prime congruent to $2$ or $5 \bmod 9$.  Similar to  the case of $3^{18}m$, any single primes congruent to $2$ or $5 \bmod 9$ dividing $m$ must come from $F(1,1,1)$ and, apart from a power of $(1-\omega)$ and a unit, the remaining $F(\omega^i,\omega^j,\omega^k)$ can contain only primes $p\equiv -1 \bmod 9$ or factors of  $p\equiv 1 \bmod 9$ in $\mathbb{Z}[\omega]$, and can be written in the form $1+3(1-\omega)(a\omega+b)$. This time we are not guaranteed $3\mid a+b$.  To add an extra 3 over $3^{18}$, we either have (Case~1) an extra $3\mid F(1,1,1)$, or (Case~2) an extra $(1-\omega)$ divides one of the $F(1,\omega^j,\omega^k)$ and its conjugate.
For Case~1, we can assume that  $3^2\parallel F(1,1,1)$ and for the remaining values that $3\parallel F(1,\omega^j,\omega^k)$ and $(1-\omega)\parallel F(\omega,\omega^j,\omega^k)$.
For Case~2, we can assume that $3\parallel F(1,1,1)$, $3(1-\omega)\parallel F(1,1,\omega)$, $3\parallel F(1,\omega,\omega^k)$, and $(1-\omega)\parallel F(\omega,\omega^j,\omega^k)$. Now obtaining integer coefficients from \eqref{RecoverCoeff} depends only on the values of the $F(\omega^i,\omega^j,\omega^k) \bmod 27$ and, since we are summing over full sets of conjugates, it is enough to check the values mod $9(1-\omega)$.
Notice that 
\begin{align*} (1-\omega) \left(1+3(1-\omega)(a\omega+b)\right)&\equiv 1-\omega -9(a+b) \text{ mod } 9(1-\omega),\\
3 \left(1+3(1-\omega)(a\omega+b)\right) & \equiv 3 \text{ mod } 9(1-\omega). 
\end{align*}
Hence, allowing for the three possibilities $a+b\equiv 0,\pm 1 \bmod 3$, and replacing $F$ by $-F$ as necessary, we can assume in Case~1 that
\[
F(1,1,1)=9, \quad
F(1,\omega,\omega^k)=3u_k, \; k=0,1,2, \quad F(1,1,\omega)=3u,
\]
and
\[
F(\omega,\omega^j,\omega^k)=(1-\omega)u_{jk} \text{ or } -(8+\omega)u_{jk} \text{ or } (10-\omega) u_{jk}, \;\; j,k=0,1,2,
\]
where the $u,u_i,u_{jk}$ are units.
Similarly for Case~2 we can reduce to
\[
F(1,1,1)=3n,\;n=1,2 \text{ or } 4, \;\;
F(1,\omega,\omega^k)=3u_k, \; k=0,1,2,\;\; F(1,1,\omega)=3(1-\omega)u,
\]
and
\[
F(\omega,\omega^j,\omega^k)=(1-\omega)u_{jk} \text{ or } -(8+\omega)u_{jk} \text{ or } (10-\omega) u_{jk}, \;\; j,k=0,1,2.
\]
Again we can assume that $u_{00},u_0,u$ are $\pm 1$ and that the $u_{jk}$ are all in $\mathscr{U}_1$ or all in $\mathscr{U}_2$.
In Case~1 we can assume that two of $u,u_0,u_1,u_2$ are in $\mathscr{U}_1$ and two are in $\mathscr{U}_2$.
In Case~2 when $n=1$ or $4$  we can assume that two of $u_0,u_1,u_2$ are $\mathscr{U}_1$ and one is in $\mathscr{U}_2$ and when $n=2$ we must have one in $\mathscr{U}_1$ and two in $\mathscr{U}_2$. This produces
$2\cdot 3^8\cdot 3^9\cdot \binom{4}{2} \cdot 3^2=2^2\cdot 3^{20}$ possibilities to check in Case~1 and $3\cdot 2\cdot 3^8\cdot 3^9 \cdot 2\cdot 3\cdot 3^2=2^2\cdot 3^{21}$ in Case~2.
An exhaustive computation found that none of these $55,788,550,416$ choices produced all integer $a_{IJK}$. 
\end{proof}

\bibliographystyle{amsplain}

\begin{bibdiv}
\begin{biblist}

\bib{BPP24}{article}{
   author={Bautista Serrano, Humberto},
   author={Paudel, Bishnu},
   author={Pinner, Christopher},
   title={The integer group determinants do not determine the group},
   journal={Comb. Number Theory},
   volume={13},
   date={2024},
   number={1},
   pages={59--65},
   issn={2996-2196},
   review={\MR{4716468}},
   doi={10.2140/cnt.2024.13.59},
}

\bib{BPP25}{article}{
   author={Bautista Serrano, Humberto},
   author={Paudel, Bishnu},
   author={Pinner, Christopher},
   title={The integer group determinants for $\mathrm{GA}(1,p)$ and related semidirect products},
   date={11 Jan 2025},
   note={2401.02657v2 [math.NT]},
}

\bib{BoerkoelPinner}{article}{
   author={Boerkoel, Ton},
   author={Pinner, Christopher},
   title={Minimal group determinants and the Lind-Lehmer problem for dihedral groups},
   journal={Acta Arith.},
   volume={186},
   date={2018},
   number={4},
   pages={377--395},
   issn={0065-1036},
   review={\MR{3879399}},
   doi={10.4064/aa180708-30-8},
}

\bib{CP3}{article}{
   author={Clem, Stian},
   author={Pinner, Christopher},
   title={The Lind-Lehmer constant for 3-groups},
   journal={Integers},
   volume={18},
   date={2018},
   pages={Paper No. A40, 20},
   review={\MR{3794030}},
}

\bib{pgroups}{article}{
   author={De Silva, Dilum},
   author={Mossinghoff, Michael J.},
   author={Pigno, Vincent},
   author={Pinner, Christopher},
   title={The Lind-Lehmer constant for certain $p$-groups},
   journal={Math. Comp.},
   volume={88},
   date={2019},
   number={316},
   pages={949--972},
   issn={0025-5718},
   review={\MR{3882290}},
   doi={10.1090/mcom/3350},
}

\bib{DeSilvaPinner}{article}{
   author={DeSilva, Dilum},
   author={Pinner, Christopher},
   title={The Lind Lehmer constant for $\mathbb{Z}_p^n$},
   journal={Proc. Amer. Math. Soc.},
   volume={142},
   date={2014},
   number={6},
   pages={1935--1941},
   issn={0002-9939},
   review={\MR{3182012}},
   doi={10.1090/S0002-9939-2014-11954-X},
}

\bib{FormanekSibley}{article}{
   author={Formanek, Edward},
   author={Sibley, David},
   title={The group determinant determines the group},
   journal={Proc. Amer. Math. Soc.},
   volume={112},
   date={1991},
   number={3},
   pages={649--656},
   issn={0002-9939},
   review={\MR{1062831}},
   doi={10.2307/2048685},
}

\bib{Kaiblinger10}{article}{
   author={Kaiblinger, Norbert},
   title={On the Lehmer constant of finite cyclic groups},
   journal={Acta Arith.},
   volume={142},
   date={2010},
   number={1},
   pages={79--84},
   issn={0065-1036},
   review={\MR{2601051}},
   doi={10.4064/aa142-1-7},
}

\bib{Kaiblinger12}{article}{
   author={Kaiblinger, Norbert},
   title={Progress on Olga Taussky-Todd's circulant problem},
   journal={Ramanujan J.},
   volume={28},
   date={2012},
   number={1},
   pages={45--60},
   issn={1382-4090},
   review={\MR{2914452}},
   doi={10.1007/s11139-011-9354-6},
}

\bib{Laquer}{article}{
   author={Laquer, H. Turner},
   title={Values of circulants with integer entries},
   conference={
      title={A collection of manuscripts related to the Fibonacci sequence},
   },
   book={
      publisher={Fibonacci Assoc., Santa Clara, CA},
   },
   date={1980},
   pages={212--217},
   review={\MR{0624127}},
}

\bib{Lehmer}{article}{
   author={Lehmer, D. H.},
   title={Factorization of certain cyclotomic functions},
   journal={Ann. of Math. (2)},
   volume={34},
   date={1933},
   number={3},
   pages={461--479},
   issn={0003-486X},
   review={\MR{1503118}},
   doi={10.2307/1968172},
}

\bib{Lind}{article}{
   author={Lind, Douglas},
   title={Lehmer's problem for compact abelian groups},
   journal={Proc. Amer. Math. Soc.},
   volume={133},
   date={2005},
   number={5},
   pages={1411--1416},
   issn={0002-9939},
   review={\MR{2111966}},
   doi={10.1090/S0002-9939-04-07753-6},
}

\bib{Mahoney}{article}{
   author={Mahoney, Michael K.},
   title={Determinants of integral group matrices for some nonabelian $2$-generator groups},
   journal={Linear and Multilinear Algebra},
   volume={11},
   date={1982},
   number={2},
   pages={189--201},
   issn={0308-1087},
   review={\MR{0650732}},
   doi={10.1080/03081088208817443},
}

\bib{MahoneyNewman}{article}{
   author={Mahoney, Michael K.},
   author={Newman, Morris},
   title={Determinants of abelian group matrices},
   journal={Linear and Multilinear Algebra},
   volume={9},
   date={1980},
   number={2},
   pages={121--132},
   issn={0308-1087},
   review={\MR{0590367}},
   doi={10.1080/03081088008817358},
}

\bib{MPP19}{article}{
   author={Mossinghoff, Michael J.},
   author={Pigno, Vincent},
   author={Pinner, Christopher},
   title={The Lind-Lehmer constant for $\mathbb{Z}_2^r\times\mathbb{Z}_4^s$},
   journal={Mosc. J. Comb. Number Theory},
   volume={8},
   date={2019},
   number={2},
   pages={151--162},
   issn={2220-5438},
   review={\MR{3959883}},
   doi={10.2140/moscow.2019.8.151},
}

\bib{MP23}{article}{
   author={Mossinghoff, Michael J.},
   author={Pinner, Christopher},
   title={Prime power order circulant determinants},
   journal={Illinois J. Math.},
   volume={67},
   date={2023},
   number={2},
   pages={333--362},
   issn={0019-2082},
   review={\MR{4593894}},
   doi={10.1215/00192082-10596890},
}

\bib{Heisenberg}{article}{
   author={Mossinghoff, Michael J.},
   author={Pinner, Christopher},
   title={The integer group determinants for the Heisenberg group of order $p^3$},
   journal={Michigan Math. J.},
   volume={74},
   date={2024},
   number={3},
   pages={551--569},
   issn={0026-2285},
   review={\MR{4767505}},
   doi={10.1307/mmj/20216124},
}

\bib{Newman1}{article}{
   author={Newman, Morris},
   title={On a problem suggested by Olga Taussky-Todd},
   journal={Illinois J. Math.},
   volume={24},
   date={1980},
   number={1},
   pages={156--158},
   issn={0019-2082},
   review={\MR{0550657}},
}

\bib{Newman2}{article}{
   author={Newman, Morris},
   title={Determinants of circulants of prime power order},
   journal={Linear and Multilinear Algebra},
   volume={9},
   date={1980},
   number={3},
   pages={187--191},
   issn={0308-1087},
   review={\MR{0601702}},
   doi={10.1080/03081088008817367},
}

\bib{OP26}{article}{
   author={Ostergaard, Andrew},
   author={Pinner, Christopher},
   title={The integer group determinants for $\mathrm{GA}(1,q)$},
   date={8 Sep 2026},
   note={arXiv:2501.07037v2 [math.NT]},
}

\bib{Panraksa25}{article}{
   author={Panraksa, Chatchawan},
   title={The 5-divisible integer group determinants for the elementary abelian group of order 25},
   journal={Ramanujan J.},
   volume={70},
   date={2026},
   note={Article 66},
   pages={6 pp.},
   doi={10.1007/s11139-026-01440-3},
}

\bib{Panraksa49}{article}{
   author={Panraksa, Chatchawan},
   title={Integer group determinants for the elementary abelian group of order 49},
   date={8 Sep 2026},
   note={arXiv:2609.09542 [math.NT]},
   pages={19 pp.}
}

\bib{Panraksa24}{article}{
   author={Panraksa, Chatchawan},
   title={Integer group determinants for abelian groups of order 24},
   date={9 Sep 2026},
   note={arXiv:2609.10423 [math.NT]},
   pages={30 pp.}
}

\bib{PP15}{article}{
   author={Paudel, Bishnu},
   author={Pinner, Christopher},
   title={Integer circulant determinants of order 15},
   journal={Integers},
   volume={22},
   date={2022},
   pages={paper no. A4, 21 pp.},
   review={\MR{4363104}},
}

\bib{dicyclic}{article}{
   author={Paudel, Bishnu},
   author={Pinner, Christopher},
   title={Minimal group determinants for dicyclic groups},
   journal={Mosc. J. Comb. Number Theory},
   volume={10},
   date={2021},
   number={3},
   pages={235--248},
   issn={2220-5438},
   review={\MR{4313424}},
   doi={10.2140/moscow.2021.10.235},
}

\bib{PaudelPinnerZnxH}{article}{
   author={Paudel, Bishnu},
   author={Pinner, Christopher},
   title={The group determinants for $\mathbb{Z}_n \times H$},
   journal={Notes Number Theory Discrete Math.},
   volume={29},
   number={3},
   date={2023},
   pages={603--619}, 
}

\bib{PPQ16}{article}{
   author={Paudel, Bishnu},
   author={Pinner, Christopher},
   title={The integer group determinants for $Q_{16}$},
   date={22 Feb 2023},
   note={arXiv:2302.11688 [math.NT]},
   pages={5 pp.},
}

\bib{PaudelPinner25}{article}{
   author={Paudel, Bishnu},
   author={Pinner, Chris},
   title={The integer group determinants for the groups of order 18},
   journal={Ramanujan J.},
   volume={68},
   date={2025},
   number={4},
   pages={paper no. 102, 23 pp.},
   issn={1382-4090},
   review={\MR{4989500}},
   doi={10.1007/s11139-025-01249-6},
}

\bib{Pigno1}{article}{
   author={Pigno, Vincent},
   author={Pinner, Christopher},
   title={The Lind-Lehmer constant for cyclic groups of order less than 892,371,480},
   journal={Ramanujan J.},
   volume={33},
   date={2014},
   number={2},
   pages={295--300},
   issn={1382-4090},
   review={\MR{3165542}},
   doi={10.1007/s11139-012-9443-1},
}

\bib{PPV16}{article}{
   author={Pigno, Vincent},
   author={Pinner, Christopher},
   author={Vipismakul, Wasin},
   title={The Lind-Lehmer constant for $\mathbb{Z}_m\times\mathbb{Z}^n_p$},
   journal={Integers},
   volume={16},
   date={2016},
   pages={Paper No. A46, 12},
   review={\MR{3522826}},
}

\bib{PinnerS4}{article}{
   author={Pinner, Christopher},
   title={The integer group determinants for the symmetric group of degree four},
   journal={Rocky Mountain J. Math.},
   volume={49},
   date={2019},
   number={4},
   pages={1293--1305},
   issn={0035-7596},
   review={\MR{3998922}},
   doi={10.1216/RMJ-2019-49-4-1293},
}

\bib{smallgps}{article}{
   author={Pinner, Christopher},
   author={Smyth, Christopher},
   title={Integer group determinants for small groups},
   journal={Ramanujan J.},
   volume={51},
   date={2020},
   number={2},
   pages={421--453},
   issn={1382-4090},
   review={\MR{4056860}},
   doi={10.1007/s11139-018-0092-x},
}

\bib{OTT}{article}{
    author={Olga Taussky Todd},
    title={Integral group matrices},
    journal={Notices Amer. Math. Soc.},
    volume={24},
    number={3},
    date={1977},
    pages={A-345},
    note={Abstract no. 746-A15, 746th Meeting, Hayward, CA, Apr. 22--23, 1977},
}

\bib{YYLaquer23}{article}{
   author={Yamaguchi, Naoya},
   author={Yamaguchi, Yuka},
   title={Remark on Laquer's theorem for circulant determinants},
   journal={Int. J. Group Theory},
   volume={12},
   date={2023},
   number={4},
   pages={265--269},
   issn={2251-7650},
   review={\MR{4526227}},
   doi={10.22108/IJGT.2022.133217.1791},
}

\bib{YYC8C2}{article}{
   author={Yamaguchi, Naoya},
   author={Yamaguchi, Yuka},
   title={Generalized Dedekind's theorem and its application to integer group determinants},
   journal={J. Math. Soc. Japan},
   volume={76},
   date={2024},
   number={4},
   pages={1123--1138},
   issn={0025-5645},
   review={\MR{4814714}},
   doi={10.2969/jmsj/90539053},
}

\bib{YYCirc16}{article}{
   author={Yamaguchi, Yuka},
   author={Yamaguchi, Naoya},
   title={Integer circulant determinants of order 16},
   journal={Ramanujan J.},
   volume={61},
   date={2023},
   number={4},
   pages={1283--1294},
   issn={1382-4090},
   review={\MR{4613461}},
   doi={10.1007/s11139-022-00599-9},
}

\bib{YYC24}{article}{
   author={Yamaguchi, Yuka},
   author={Yamaguchi, Naoya},
   title={Integer group determinants for $\rm C_2^4$},
   journal={Ramanujan J.},
   volume={62},
   date={2023},
   number={4},
   pages={983--995},
   issn={1382-4090},
   review={\MR{4667334}},
   doi={10.1007/s11139-023-00727-z},
}

\bib{YYC422}{article}{
   author={Yamaguchi, Yuka},
   author={Yamaguchi, Naoya},
   title={Integer group determinants for abelian groups of order 16},
   journal={Hiroshima Math. J.},
   volume={54},
   date={2024},
   number={3},
   pages={359--373},
   issn={0018-2079},
   review={\MR{4822670}},
   doi={10.32917/h2023012},
}

\bib{YYC42}{article}{
   author={Yamaguchi, Yuka},
   author={Yamaguchi, Naoya},
   title={Integer group determinants for $\rm C^2_4$},
   journal={Integers},
   volume={24},
   date={2024},
   pages={Paper No. A30, 21},
   review={\MR{4718924}},
   doi={10.5281/zenodo.10821710},
}

\bib{YYNA16}{article}{
   author={Yamaguchi, Yuka},
   author={Yamaguchi, Naoya},
   title={Integer group determinants for three of the non-abelian groups of order 16},
   journal={Res. Number Theory},
   volume={10},
   date={2024},
   number={2},
   pages={paper no. 23, 19 pp.},
   issn={2522-0160},
   review={\MR{4713837}},
   doi={10.1007/s40993-024-00508-7},
}

\bib{YY16}{article}{
   author={Yamaguchi, Yuka},
   author={Yamaguchi, Naoya},
   title={Integer group determinants of order 16},
   journal={Ramanujan J.},
   volume={65},
   date={2024},
   number={3},
   pages={1459--1474},
   issn={1382-4090},
   review={\MR{4810521}},
   doi={10.1007/s11139-024-00946-y},
}

\end{biblist}
\end{bibdiv}

\end{document}